\documentclass{amsart}
\usepackage{setspace, amsmath, amsthm, amssymb, amsfonts, amscd, epic, graphicx, ulem, dsfont}
\usepackage[T1]{fontenc}
\usepackage{multirow}
\usepackage{bbm}
\usepackage{enumerate}

\makeatletter \@namedef{subjclassname@2010}{
  \textup{2010} Mathematics Subject Classification}
\makeatother

\newtheorem{thm}{Theorem}[section]

\newtheorem{cor}[thm]{Corollary}

\newtheorem{pro}[thm]{Proposition}

\theoremstyle{remark}
\newtheorem*{rema}{Remark}

\theoremstyle{definition}

\newtheorem{exa}[thm]{\textbf{Example}}

\newcommand{\Ima}{\text{\rm{Im}}}
\newcommand{\re}{\text{\rm{Re}}}
\newcommand{\tr}{\text{\rm{tr}}}

\newcommand{\R}{\mathbb{R}}

\newcommand{\N}{\mathbb{N}}

\newcommand{\C}{\mathbb{C}}

\begin{document}

\title[Nilpotence Implying Normality]{When Nilpotence Implies Normality of Bounded Linear Operators}
\author[N. FRID, M. H. MORTAD]{Nassima Frid and Mohammed Hichem Mortad$^*$}

\dedicatory{}
\thanks{* Corresponding author.}
\date{}


\address{(The first author): Département de
Mathématiques, Université Oran1, Ahmed Ben Bella, B.P. 1524,
El Menouar, Oran 31000, Algeria.}

\email{nassima.frid@yahoo.fr}

\address{(The corresponding author) Department of
Mathematics, University of Oran 1, Ahmed Ben Bella, B.P. 1524, El
Menouar, Oran 31000, Algeria.}

\email{mhmortad@gmail.com, mortad.hichem@univ-oran1.dz.}

\begin{abstract}
In this paper, we give conditions forcing nilpotent matrices (and
bounded linear operators in general) to be null or equivalently to
be normal. Therefore, a non-zero operator having e.g. a positive
real part is never nilpotent. The case of quasinilpotence is also
considered.
\end{abstract}

\maketitle

\section{Introduction}

Let $H$ be a Hilbert space and let $B(H)$ be the algebra of all
bounded linear operators defined from $H$ into $H$. Recall that
$T\in B(H)$ is said to be positive, symbolically $T\geq0$, if
$<Tx,x>\geq0$ for all $x\in H$. Recall also that any $T$ may always
be expressed as $T=A+iB$ with $A,B\in B(H)$ being both self-adjoint
and $i=\sqrt{-1}$. Necessarily, $A=(T+T^*)/2$ which will be denoted
by $\re T$ and it is called the real part of $T$. Also,
$B=(T-T^*)/{2i}$ is the imaginary part of $T$, written $\Ima T$.

As is well known, nilpotent matrices play an important role in
matrix theory, and in operator theory in general. The following was
shown in \cite{Mortad-square-roots-normal}:

\begin{pro}\label{MORTAD-PROP}
If $T\in B(H)$ is such that $\re T\geq0$ and $T^2=0$, then $T=0$
\end{pro}

In this paper, we carry on this investigation and deal with the
general case.

While we assume readers are familiar with notions and results in
matrix and operator theories (see e.g. \cite{Axler LINEAR ALG DONE
RIGHT} and \cite{Mortad-Oper-TH-BOOK-WSPC}), we recall a few well
established facts. For example, if $T\in B(H)$ is normal, then
\[\|T^n\|=\|T\|^n,~\forall n\in\N.\]
It seems noteworthy to emphasize that thanks to the previous
equality, if $T$ is nilpotent then "$T=0\Leftrightarrow T\text{ is
normal}$". Therefore, when we further assume that $\re T\geq0$ and
prove Theorem \ref{Main THM} below, then this will become yet
another characterization to be added to the 89 conditions equivalent
to the normality of a matrix already obtained in
\cite{Elsner-Ikramov} and \cite{Grone et al-normal matrices}. A
somehow related paper is \cite{Fillmore et al quasinilpotent}. In
the end, we note that there is still some ongoing extensive research
on normal matrices. See e.g. \cite{Wang-Zhang} or
\cite{Gerasimova-Unitary similarity normal matrix}. See also the
interesting paper \cite{Garcia et al-Similarity finite dimensional}
about the similarity of products of normal matrices. In the end,
readers may wish to consult the informal notes \cite{Mercer} which
contain some interesting examples of nilpotent matrices.

\section{Main Results}

The following is the main result in this paper. It tells us that a
(non-zero) operator $T$ with a positive (or negative) real or
imaginary part is never nilpotent.

\begin{thm}\label{Main THM}
Let $T=A+iB\in B(H)$ and let $n\geq2$. If $T^n=0$ and $A\geq0$ (or
$B\geq 0$), then $T=0$.
\end{thm}

\begin{proof}The proof is carried out in two steps.
\begin{enumerate}
  \item Let $\dim H<\infty$. The proof uses a trace argument. First, assume that $A\geq0$.
Clearly, the nilpotence of  $T$ does yield $\tr T=0$. Hence
\[0=\tr(A+iB)=\tr A+i\tr B.\]

Since $A$ and $B$ are self-adjoint, we know that $\tr A,\tr B\in\R$.
By the above equation, this forces $\tr B=0$ and $\tr A=0$. The
positiveness of $A$ now intervenes to make $A=0$. Therefore, $T=iB$
and so $T$ is normal. Thus, and as alluded above,
\[0=\|T^n\|=\|T\|^n,\]
thereby, $T=0$.

In the event $B\geq0$, reason as above to obtain $T=A$ and so $T=0$,
as wished.
  \item Let $\dim H=\infty$. The condition $\re T\geq0$ is
  equivalent to $\re<Tx,x>\geq0$ for all $x\in H$. So if $E$ is a closed invariant
  subspace of $T$, then the previous condition also holds for
  $T|E:E\to E$.

  Now, we proceed to show that $T=0$, i.e. we must show that $Tx=0$ for all $x\in H$. So, let $x\in
  H$ and let $E$ be the span of $x, Tx,\cdots, T^{n-1}x$ (that is, the orbit of $x$ under the action of $T$). Hence $E$ is a finite dimensional subspace of $H$ (and so it is equally a Hilbert space).
  By the nilpotence assumption, we have
  \[T^nx=0,\]
  from which it follows that $E$ is invariant for $T$. So, by the first part
  of the proof (the finite dimensional case), we know that $T=0$ on
  $E$ whereby $Tx=0$. As this holds for any $x$, it follows that
  $T=0$ on $H$, as needed.
\end{enumerate}
\end{proof}

\begin{rema}For example, the condition $A\geq0$ may not just be
dropped. Indeed, if $T=\left(
                         \begin{array}{cc}
                           0 & 1 \\
                           0 & 0 \\
                         \end{array}
                       \right)$, then $T^2=0$ but $T\neq0$. Observe
                       finally
                       that
\[A=\re T=\frac{1}{2}\left(
                       \begin{array}{cc}
                         0 & 1 \\
                         1 & 0 \\
                       \end{array}
                     \right)
\]
is neither positive nor negative for $\sigma(A)=\{-1/2,1/2\}$.
\end{rema}

Mutatis mutandis, we may also show the following result.

\begin{pro}
Let $T$ be a finite square matrice decomposed as $T=A+iB$ (where $A$
and $B$ are self-adjoint). If $T^n=0$ and $A\leq0$ or $B\leq 0$,
then $T=0$.
\end{pro}

\begin{rema}
As mentioned above, the power of Theorem \ref{Main THM} lies in the
fact it easily allows us to test the non-nilpotence of a given
operator. For example, if $V$ is the Volterra's operator defined on
$L^2(0,1)$, that is,
\[Vf(x)=\int_0^xf(t)dt,~f\in L^2(0,1).\]
Then, it well known that $V$ is not nilpotent. Let's corroborate
this fact using Theorem \ref{Main THM}. Since $\re V\geq0$ (see e.g.
Exercise 9.3.21 in \cite{Mortad-Oper-TH-BOOK-WSPC}), assuming the
nilpotence of $V$ would make $V=0$, and this is impossible. Thus,
$V$ is not nilpotent.
\end{rema}

Here is an alternative and interesting reformulation of Theorem
\ref{Main THM} over finite dimensional spaces.

\begin{cor}
Let $T\in M_n(\C)$ be nilpotent (with $T\neq0$). Then $(T+T^*)/2$
(or $(T-T^*)/{2i}$) has at least two eigenvalues of opposite signs.
\end{cor}

\begin{proof}If the self-adjoint $(T+T^*)/2$ has only positive (or
only negative) eigenvalues, then the nilpotence of $T$ would yield
$T=0$, contradicting the other assumption. Reason similarly with
$(T-T^*)/{2i}$.
\end{proof}

\begin{rema}
It is well known that a nilpotent operator $T$ necessarily has a
spectrum reduced to the singleton $\{0\}$ (operators with this
property are called quasinilpotent). As readers are already wary,
the concepts of nilpotence and quasinilpotence do coincide on finite
dimensional vector spaces. So, it is legitimate to wonder whether,
in the infinite dimensional setting, the nilpotence's assumption may
be weakened to requiring $\sigma(T)=\{0\}$ only? The answer is no by
considering again the Volterra's operator. Indeed, if $V$ designates
the Volterra's operator on $L^2(0,1)$, then it well known that
$\sigma(V)=\{0\}$ (i.e. $V$ is quasinilpotent) and that $\re V\geq0$
and yet $V\neq0$.
\end{rema}

In many results in operator theory, the asymmetric condition
$\sigma(A)\cap \sigma(-A)\subseteq \{0\}$ yields similar conclusions
as when assuming the positivity (or negativity) of $A$. It is also
known that this asymmetric condition is weaker that positiveness
(and negativeness) of $A$.

Nonetheless, we have the following result.

\begin{thm}\label{asymmetric spectrum dim leq 3 THM}Let $H$ be a Hilbert space of dimension $k$ where $k=2$ or $k=3$. Let $T=A+iB\in B(H)$
be nilpotent. If $\sigma(A)\cap \sigma(-A)=\{0\}$ or $\sigma(B)\cap
\sigma(-B)=\{0\}$, then $T=0$.
\end{thm}

\begin{proof}\hfill
\begin{enumerate}
  \item Let $k=2$. As above, we may obtain that $\tr A=0$. Since $A$ is self-adjoint, it follows that $A$ is similar to $\left(
                                                                                                                              \begin{array}{cc}
                                                                                                                                \alpha & 0 \\
                                                                                                                                0 & -\alpha \\
                                                                                                                              \end{array}
                                                                                                                            \right)$ where $\alpha\in\R$. So, if $\alpha\neq0$, then $\sigma(A)\cap
                                                                                                                            \sigma(-A)=\{0\}$ will be
                                                                                                                            violated.
                                                                                                                            Thence,
                                                                                                                            $\alpha=0$,
                                                                                                                            that
                                                                                                                            is,
                                                                                                                            $A=0$.
                                                                                                                            Consequently,
                                                                                                                            we
                                                                                                                            obtain
                                                                                                                            $T=0$
                                                                                                                            as
                                                                                                                            above.
                                                                                                                            The
                                                                                                                            corresponding case for $B$ can be dealt with similarly.
  \item Assume now that $k=3$. If $\sigma(A)\cap \sigma(-A)=\{0\}$,
  then in view of the self-adjointness of $A$, we know that $A$ is
  similar to $\left(
                \begin{array}{ccc}
                  0 & 0 & 0 \\
                  0 & \alpha & 0 \\
                  0 & 0 & -\alpha \\
                \end{array}
              \right)$ (where $\alpha\in\R$) given that $\tr A=0$ and
              $0\in\sigma(A)$. As before, we mus necessarily have
              $\alpha=0$ and so $A=0$. The nilpotence of $T=iB$ then
              gives $T=0$.
\end{enumerate}
\end{proof}

\begin{rema}
When $\dim H=4$, a similar idea is seemingly not applicable. Indeed,
a self-adjoint $4\times 4$ matrix $A$ such that $\tr A=0$ and
$\sigma(A)\cap\sigma(-A)=\{0\}$ may be non-null. For example, take
\[A=\left(
      \begin{array}{cccc}
        0 & 0 & 0 &0\\
        0 & 3 & 0 & 0\\
        0 & 0 & -2 & 0\\
        0 &  0 &  0 & -1
      \end{array}
    \right).
\]
Then clearly $\tr A=0$ and
\[\sigma(A)\cap \sigma(-A)= \{0\},\]
yet $A\neq 0$. This does not exclude the possibility of a proof when
$\dim H=4$. Let us therefore give a counterexample:
\end{rema}

\begin{exa}\label{pp}Take
\[T=\left(
      \begin{array}{cccc}
        2 & 2 & -2 & 0 \\
        5 & 1 & -3 & 0 \\
        1 & 5 & -3 & 0 \\
        0 & 0 & 0 & 0 \\
      \end{array}
    \right)\text{ and so } A=\left(
      \begin{array}{cccc}
        2 & 7/2 & -1/2 & 0 \\
        7/2 & 1 & 1 & 0 \\
        -1/2 & 1 & -3 & 0 \\
        0 & 0 & 0 & 0 \\
      \end{array}
    \right).\]
Hence (approximatively)
\[\sigma(A)=\{0, -3.71,-1.33,5.04\}\]
and so $\sigma(A)\cap \sigma(-A)=\{0\}$ is trivially satisfied.
Observe finally that $T\neq0$ whereas $T^3=0$, i.e. $T$ is
nilpotent.
\end{exa}

We may easily prove the following result.

\begin{pro}Let $H$ be a 4-dimensional Hilbert space. Let $T=A+iB\in B(H)$
be nilpotent. If $\sigma(A)\cap \sigma(-A)=\{0\}$ with 0 being an
eigenvalue of multiplicity 2, then $T=0$.
\end{pro}

\begin{proof}
Just write
\[A\sim \left(
      \begin{array}{cccc}
        0 & 0 & 0 &0\\
        0 & 0 & 0 & 0\\
        0 & 0 & \alpha & 0\\
        0 &  0 &  0 & -\alpha
      \end{array}
    \right)\]
and obtain $A=0$. Hence $T=iB$ and so $T=0$, as above.
\end{proof}

We also have the following related result.

\begin{pro}
If $A$ is a self-adjoint $2\times2$ matrix such that $\sigma(A)\cap
\sigma(-A)=\varnothing$, then $T=A+iB$ is never nilpotent.
\end{pro}

\begin{proof}
If $T$ were nilpotent, then $\tr A=0$. This would necessarily make
$A$ look like $\left(
   \begin{array}{cc}
     \alpha & 0 \\
     0 & -\alpha \\
   \end{array}
 \right)
$ (with $\alpha\in\R$). This condition is, however, not consistent
with $\sigma(A)\cap \sigma(-A)=\varnothing$. Thus, $T$ cannot be
nilpotent.
\end{proof}

Finally, notice that there are nilpotent matrices $T=A+iB$ of higher
order such that $\sigma(A)\cap \sigma(-A)=\varnothing$. We may just
consider the non-zero block matrix from Example \ref{pp}. Otherwise,
we give another example:

\begin{exa}Let
\[T=\left(
      \begin{array}{cccc}
        0 & 1 & 2 & 4 \\
        0 & 0 & 2 & 1 \\
        0 & 0 & 0 & 5 \\
        0 & 0 & 0 & 0 \\
      \end{array}
    \right)\text{ and so } A=\left(
      \begin{array}{cccc}
        0 & 1/2 & 1 & 2 \\
        1/2 & 0 & 1 & 1/2 \\
        1 & 1 & 0 & 5/2 \\
        2 & 1/2 & 5/2 & 0 \\
      \end{array}
    \right).\]
Then, it may be checked that the eigenvalues of $A$ are
approximatively $+4.058$, $-1.043$, $-2.811$ and $-0.205$. Hence the
requirement $\sigma(A)\cap \sigma(-A)=\varnothing$ is clearly
satisfied. In the end, $T^4=0$ (with $T^3\neq0$).

\end{exa}

\section{Acknowledgement}

The authors wish to express their gratitude to Professor Alexander
M. Davie (the University of Edinburgh, UK) for his help in the
infinite dimensional case of the proof of Theorem \ref{Main THM}.

\end{document}